\documentclass[11pt]{amsart}

\usepackage{amsmath,amssymb,amsthm,mathtools}
\usepackage[margin=1.2in]{geometry}
\usepackage{hyperref}

\newtheorem{theorem}{Theorem}[section]
\newtheorem{proposition}[theorem]{Proposition}
\newtheorem{lemma}[theorem]{Lemma}
\newtheorem{corollary}[theorem]{Corollary}
\theoremstyle{remark}
\newtheorem{remark}[theorem]{Remark}

\theoremstyle{definition}
\newtheorem{definition}[theorem]{Definition}

\newcommand{\Z}{\mathbb{Z}}
\newcommand{\R}{\mathbb{R}}
\newcommand{\C}{\mathbb{C}}
\newcommand{\Q}{\mathbb{Q}}
\newcommand{\OO}{\mathcal{O}}

\DeclareMathOperator{\Gal}{Gal}

\title[Monogenity of $(k,t)$-Fibonacci and $(k,t)$-Lucas polynomials]
      {Monogenity of the irreducible factors of\\ $(k,t)$-Fibonacci and $(k,t)$-Lucas
       polynomials}

\author{Michail Karatarakis}
\address[{Michail Karatarakis}]{Radboud University, Nijmegen, The Netherlands
}
\email{michail.karatarakis@ru.nl}

\author{Sumandeep Kaur}
\address[Sumandeep Kaur\footnote{Corresponding author}]{Department of Mathematics, Shanghai University, China} 
\email{suman@shu.edu.cn}
\date{}
\subjclass[2020]{11R04, 11R09, 11B39, 11C08}
\keywords{Monogenity, Fibonacci polynomials,
Lucas polynomials, $(k,t)$-Fibonacci polynomials, $(k,t)$-Lucas polynomials,
Dickson polynomials, cyclotomic fields, Chebyshev polynomials}

\begin{document}

\begin{abstract}
In this article, we study the monogenity of the number field generated by a root of an irreducible factor of the generalized $(k,t)$-Fibonacci and $(k,t)$-Lucas polynomial sequence. We use Dickson polynomials and their connection with cyclotomic fields to prove our main results. As special cases, we recover the monogenity results for the classical Fibonacci and Lucas polynomials and obtain corresponding monogenity results for Chebyshev polynomials of the first and second kinds.

\end{abstract}

\maketitle

\section{Introduction}\label{sec:intro}

Recall that an algebraic number field $K$ is said to be {monogenic} if $\OO_K=\Z[\theta]$ for some $\theta\in\OO_K$. In this case, $\{1,\theta,\dots,\theta^{[K:\Q]-1}\}$ is an integral basis of $K$; such an integral basis of $K$ is called a {power basis} of $K$. A monic irreducible polynomial $f\in\Z[X]$ is
said to be {monogenic} if it has a root $\alpha$ with $\Z[\alpha]=\OO_{\Q(\alpha)}$. Several classes of monogenic polynomials are known and their properties have been studied in the literature (\cite{Gaal}-\cite{HJ}, \cite{LJ}, \cite{LJ1}, \cite{LJ2}, \cite{LJJ}, \cite{Le1}, \cite{joo}, \cite {HS1}).

The Fibonacci polynomials $\{F_n\}_{n\geq 0}$ and Lucas polynomials
$\{L_n\}_{n\geq 0}$ are defined recursively by
$F_0(X)=0,~ F_1(X)=1,~
F_{n+2}(X)=X\,F_{n+1}(X)+F_n(X),~ n\geq 0,$
and
$L_0(X)=2,~ L_1(X)=X,~
L_{n+2}(X)=X\,L_{n+1}(X)+L_n(X),~ n\geq 0,
$
respectively. These polynomials  are one of the most important families of recursively defined polynomial sequences. They has a wide range of  algebraic and arithmetic properties and due to this, these polynomials have  applications in several areas of mathematics, like number theory, combinatorics, algebra, approximation theory, and the theory of linear recurrence sequences.  Their irreducibility, factorization, zeros, divisibility properties, discriminants, and Galois groups have been investigated by several mathematicians (see \cite{HB}). Recently, Chen, Guo and Hong \cite{CGH} studied the monogenity of these polynomials. Precisely they proved:

\begin{theorem}[\cite{CGH}]\label{thm:CGH}
For every odd $n$, all irreducible factors of $F_n$ are monogenic.  For every even $n$, all
irreducible factors of $L_n$ are monogenic.
\end{theorem}

Falcon and Plaza \cite{FP} introduced the following two-parameter generalization of classical Fibonacci and Lucas polynomails:
\begin{equation}\label{eq:ktdef}
\begin{aligned}
F^{(k,t)}_0&=0, &
F^{(k,t)}_1&=1, &
F^{(k,t)}_{n+2}&=kXF^{(k,t)}_{n+1}+tF^{(k,t)}_n,\\
L^{(k,t)}_0&=2, &
L^{(k,t)}_1&=kX, &
L^{(k,t)}_{n+2}&=kXL^{(k,t)}_{n+1}+tL^{(k,t)}_n,
\end{aligned}
\end{equation}
where $k,t\in\mathbb{Z}$. Theorem~\ref{thm:CGH} naturally motivates the study of monogenity for  families of these extended Fibonacci and extended Lucas polynomials. 

Note that the classical Fibonacci polynomials $F_n$ and Lucas polynomials $L_n$ are recovered by taking $(k,t)=(1,1)$. Consequently, Theorem~\ref{thm:CGH} corresponds precisely to this particular choice of parameters.

The aim of this paper is to investigate the monogenity of the irreducible factors of the generalized Fibonacci and Lucas polynomials defined in \eqref{eq:ktdef}. In this way, we extend the result of Chen et al. \cite{CGH}.\\

\noindent\textbf{Notations.} In what follows, all roots are taken in a fixed algebraic closure $\overline\Q\subset\C$.  For an integer
$n\ge1$ we write $\zeta_n:=e^{2\pi i/n}$, a primitive $n$-th root of unity,
$\varphi$ will stand for Euler's totient function, $\Phi_m(X)$ for the $m$-th cyclotomic polynomial, and
$\OO_K$ for the ring of integers of a number field $K$. An element $\zeta$ is a {primitive $m$-th root of unity} if it has order exactly $m$ in
$\C^\times$.  

Now, we state our main result.

\begin{theorem}\label{thm:main}
Let $k,s$ be nonzero integers and let $t=\pm s^2$.  Let $n\ge1$, and let $f$ be an irreducible factor of $F_n^{(k,t)}$ or
$L_n^{(k,t)}$. Let $x\in\overline{\Q}$ be a root of $f$ and set $ \theta:=\frac{k\,x}{s}.$
Then $\theta$ is an algebraic integer, $\Q(\theta)=\Q(x)$, and $\OO_{\Q(x)}=\Z[\theta]$.
 Moreover, there exists a root of unity $\xi$ such that $\theta$ has the form indicated in the last column of the following table. 
\begin{center}
\begin{tabular}{cccc}
\hline
$t$ & family & hypothesis on $n$ & $\theta$ equals\\
\hline
$+s^2$ & $F^{(k,t)}_n$ & $n$ odd & $\xi-\xi^{-1}$\\
$+s^2$ & $L^{(k,t)}_n$ & $n$ even, $n\ge2$ & $\xi-\xi^{-1}$\\
$-s^2$ & $F^{(k,t)}_n$ & $n\ge1$ & $\xi+\xi^{-1}$\\
$-s^2$ & $L^{(k,t)}_n$ & $n\ge1$ & $\xi+\xi^{-1}$\\
\hline
\end{tabular}
\end{center}
\end{theorem}

The following result, which is the main theorem of~\cite{CGH}, is an immediate
consequence of Theorem~\ref{thm:main}.

\begin{corollary}\label{cor:CGH}
For every odd $n$, all irreducible factors of $F_n$ are monogenic.  For every even $n$, all
irreducible factors of $L_n$ are monogenic.
\end{corollary}

\section{Preliminaries}
\label{sec:engine}

Throughout this section, $m>1$ is an integer, $\zeta$ is a primitive $m$-th root of unity,
$K=\Q(\zeta)$, and $c\in\Z$.  Recall that $\OO_K=\Z[\zeta]$, and $K/\Q$ is Galois with $\Gal(K/\Q)\cong(\Z/m\Z)^\times$.

\begin{proposition}\label{prop:main}
Let $u:=\zeta+c\,\zeta^{-1}$. Suppose that $\Q(u)$ is a proper subfield of $K$, then $\OO_{\Q(u)}=\Z[u]$.  In particular;
$\Q(u)$ is monogenic.
\end{proposition}

\begin{lemma}\label{lem:shift}
Let $N\ge0$, and let $a_0,\dots,a_N\in\Q$. Define
$$ g(X)=\sum_{k=0}^{N}a_k\,X^{\,N-k}\,(X^2+c)^{k}\ \in\Q[X].$$
Then, $\deg g\le 2N$, the coefficient of $X^{2N}$ in $g$ is $a_N$, and
\[
   g(\zeta)=\zeta^{N}\sum_{k=0}^{N}a_k\,u^{k},\qquad u=\zeta+c\zeta^{-1}.
\]
\end{lemma}

\begin{proof}
For each $0\le k\le N$, the $k$-th summand has degree at most $(N-k)+2k=N+k\le 2N$, with equality only for $k=N$.
Since the polynomial $X^2+c$ is monic, it follows that $\deg g\le2N$  and the leading coefficient of $X^{2N}$ is $a_N$. 
Next, since $\zeta u=\zeta^2+c$, we have $\zeta^{N-k}(\zeta^2+c)^k=\zeta^{N-k}\zeta^{k}u^{k}=\zeta^{N}u^{k}$.
Therefore,
\[
g(\zeta)
=\sum_{k=0}^Na_k\zeta^{N-k}(\zeta^2+c)^k
=\zeta^N\sum_{k=0}^Na_ku^k.
\]

\end{proof}

\begin{lemma}\label{lem:degree}
If $\Q(u)$ is a proper subfield of $K$, then
$2[\Q(u):\Q]\le\varphi(m).$
\end{lemma}

\begin{proof}
Write $d=[\Q(u):\Q]$. By the Tower Law,
 $[K:\Q]=[K:\Q(u)][\Q(u):\Q]=\varphi(m)$. Hence $d\mid\varphi(m)$. Since $\Q(u)\subsetneq K$, we have $d<\varphi(m)$, and hence $\varphi(m)/d\ge2$.
\end{proof}

\begin{proof}[Proof of Proposition~\ref{prop:main}]
It is enough to prove that $\OO_{\Q(u)}\subseteq\Z[u]$. Let
$d=[\Q(u):\Q]$.

We first prove the following claim.

\medskip

\noindent
{Claim.}
If $a_0,\ldots,a_N\in\Q$ with $2N<\varphi(m)$ and $\alpha=\sum_{k=0}^{N}a_ku^k$
is an algebraic integer, then $a_0,\ldots,a_N\in\Z$.

\medskip

We prove by induction on $N$. The case $N=0$ is trivial. It suffices to prove that
$a_N\in\Z$, since $\alpha-a_Nu^N$ is again an algebraic integer of the same form. Let $g$ be the polynomial defined in Lemma~\ref{lem:shift}. Then
$g(\zeta)=\zeta^N\alpha,~
\deg g\le2N,$
and the leading coefficient of $g$ is $a_N$.
Since $\zeta^N\alpha\in\OO_K=\Z[\zeta]$, there exists
$f\in\Z[X]$ with $\deg f<\varphi(m)$ and
$f(\zeta)=\zeta^N\alpha$, after replacing $f$ by its remainder modulo $\Phi_m$, if necessary.
Hence $(g-f)(\zeta)=0.$
As $\deg(g-f)<\varphi(m)=\deg\Phi_m,$
we obtain $g=f$.  Hence
$a_N\in\Z$, proving the claim.

Now let $\alpha\in\OO_{\Q(u)}$. Since
$\{1,u,\ldots,u^{d-1}\}$ is a $\Q$-basis of $\Q(u)$, we may write
\[
\alpha=\sum_{k=0}^{d-1}a_ku^k,
\qquad a_k\in\Q.
\]
By Lemma~\ref{lem:degree}, $2(d-1)<\varphi(m).$
Hence the claim applies and yields $a_k\in\Z$ for every $0\le k\le d-1$. Thus
$\alpha\in\Z[u]$, and the proof is complete.
\end{proof}

\begin{lemma}\label{lem:proper}
The field $\Q(\zeta+c\zeta^{-1})$ is a proper subfield of $\Q(\zeta)$ if and only if there exists a primitive $m$-th root of unity $\beta\ne\zeta$ such that $ \beta+c\beta^{-1}=\zeta+c\zeta^{-1}.$
\end{lemma}

\begin{proof}
Suppose that such $\beta$ exists.
Since $\beta$ and $\zeta$ are roots of $\Phi_m$, there exists
$\sigma\in\Gal(K/\Q)$ with $\sigma(\zeta)=\beta$. Hence
$\sigma(u)=u$. If $\Q(u)=K$, then $\sigma$ fixes a generator of $K$ and therefore
$\sigma=\mathrm{id}$, contradicting the fact that $\beta\ne\zeta$. 

Conversely, suppose that $\Q(u)\subsetneq K$. Since $K/\Q$ is Galois, there exists a
nontrivial automorphism $\sigma\in\Gal(K/\Q)$ which fixes $\Q(u)$ pointwise. Putting
$\beta=\sigma(\zeta)$, we have $\beta\ne\zeta$, $\beta$ is again a primitive $m$-th root
of unity, and $\beta+c\beta^{-1}
=\sigma(\zeta+c\zeta^{-1})
=\sigma(u)
=u.$ This completes the proof.

\end{proof}

\begin{proposition}\label{prop:critA}
Let $N\ge1$ and let $\alpha\in\C$ satisfy $\alpha^{2N}=-1$ and $\alpha^2\neq-1$.  Then the field $\Q(\alpha-\alpha^{-1})$ is monogenic. Precisely,
$\OO_{\Q(\alpha-\alpha^{-1})}=\Z[\alpha-\alpha^{-1}]$.
\end{proposition}

\begin{proof}
By hypothesis, we see that $\alpha$ is a root of unity. Let $m$ be its order. So,
$m\mid 4N$.  Set $\beta:=-\alpha^{-1}$.  Since $\alpha^{2N}=-1,$ we have
$\beta=\alpha^{2N-1}$.  Let $\gcd(2N-1,m)=e$. Then $e$ is odd. Hence
$\gcd(e,4)=1$ and $e\mid N$. This implies that $e$ divides both $2N$ and $2N-1$, and thus $e=1$.  Therefore,
$\beta$ is again a primitive $m$-th root of unity.  Moreover $\beta\neq\alpha$, 
because $\alpha^2\neq-1$ by hypothesis, and $\beta-\beta^{-1}=-\alpha^{-1}+\alpha$.  Keeping in mind Lemma~\ref{lem:proper}
with $c=-1$, we have $\Q(\alpha-\alpha^{-1})\neq\Q(\alpha)$. 
The proof is complete by Proposition~\ref{prop:main}.
\end{proof}

Recall that for $m>2$ the field $\Q(\zeta+\zeta^{-1})$ is the {maximal real subfield} of
$\Q(\zeta)$. It is fixed by complex conjugation $\zeta\mapsto\zeta^{-1}$, and hence it is contained in
$\R$, and as index $2$ in $\Q(\zeta)$ by Lemma~\ref{lem:proper}.

The following result proved in {\cite[Prop.~2.16]{Wash}} is an immediate consequence of Lemm~\ref{lem:proper} and Proposition~\ref{prop:main}.
\begin{corollary}\label{cor:washington}
For any integer $m>2$, the maximal real subfield
$\Q(\zeta+\zeta^{-1})$ is monogenic, with
$\OO_{\Q(\zeta+\zeta^{-1})}
=
\Z[\zeta+\zeta^{-1}].$
\end{corollary}

\begin{proof}
Take $c=1$ and $\beta=\zeta^{-1}$. Since $m>2$, we have
$\beta\ne\zeta$, and
$\beta+\beta^{-1}=\zeta+\zeta^{-1}.$
Hence the hypothesis of Lemma~\ref{lem:proper} is satisfied, and the conclusion follows immediately from Proposition~\ref{prop:main}.
\end{proof}
The following result is a quick application of Corollary~\ref{cor:washington}.
\begin{proposition}\label{prop:critB}
Let $\alpha\in\C$ be a root of unity with $\alpha^2\neq1$.  Then
$\OO_{\Q(\alpha+\alpha^{-1})}=\Z[\alpha+\alpha^{-1}]$.
\end{proposition}

The following Corollary proved in {\cite[Prop.~3.2]{CGH}};  follows from Lemma~\ref{lem:proper} and  Proposition~\ref{prop:main}. It extends \cite{NS}, where the cases $m=2^r\ge8$ and $m=4p^r$
with $p$ an odd prime are treated, for all $m>4$ and divisible by $4$. 
\begin{corollary}[]\label{cor:nakahara}
If $m>4$ and $4$ divides $ m$, then the field
$\Q(\zeta-\zeta^{-1})$ is monogenic with $\OO_{\Q(\zeta-\zeta^{-1})}
=
\Z[\zeta-\zeta^{-1}].$
\end{corollary}

\begin{proof}
Take $c=-1$ and $\beta=-\zeta^{-1}$.  Since $4$ divides $ m$, we have $\zeta^{m/2}=-1$ and
$\beta=\zeta^{m/2-1}$. Keeping in mind the fact that $\frac{m}{2}-1$ is odd, it is easy to see that $\gcd\!\left(\frac{m}{2}-1,m\right)=1.$ Thus
$\beta$ is a primitive $m$-th root of unity. Also,
$\beta-\beta^{-1}=-\zeta^{-1}+\zeta$, and $\beta\neq\zeta$. So the hypothesis of Lemma~\ref{lem:proper} is satisfied. Now, applying Proposition~\ref{prop:main}, we are done.
\end{proof}

\begin{remark}
Let $n>1$ be odd and let $\alpha=i\zeta_n$. Then $\alpha$ is a primitive
$4n$-th root of unity and
$\alpha-\alpha^{-1}
=
-i(\zeta_{2n}+\zeta_{2n}^{-1}).$
Using Corollary~\ref{cor:nakahara}, we get
\cite[Proposition~4.1]{CGH}.
\end{remark}

\section{Dickson polynomials and the \texorpdfstring{$(k,t)$}{(k,t)}-families}
\label{sec:dickson}

\begin{definition}\label{def:dickson}
Let $R$ be a commutative ring and let $a\in R$.  The {Dickson polynomials of the first and
second kind with parameter $a$} to be denoted by $D^{(1)}_n(X;a),\,D^{(2)}_n(X;a)\in R[X]$ are the sequences defined recursively by
\begin{equation}\label{eq:dickson}
\begin{aligned}
  D^{(1)}_0(X;a)&=2, & D^{(1)}_1(X;a)&=X, &
  D^{(1)}_{n+2}(X;a)&=X\,D^{(1)}_{n+1}(X;a)-a\,D^{(1)}_{n}(X;a),\\
  D^{(2)}_0(X;a)&=1, & D^{(2)}_1(X;a)&=X, &
  D^{(2)}_{n+2}(X;a)&=X\,D^{(2)}_{n+1}(X;a)-a\,D^{(2)}_{n}(X;a).
\end{aligned}
\end{equation}
For any $x\in R$ we will write $D^{(j)}_n(x;a)\in R$ for the value of this polynomial at $x$, and we
abbreviate $D^{(j)}_n$ when $a$ is clear from the context.  Note that both are monic of degree $n$ for $n\ge1$ (see \cite{LMT}).
\end{definition}


\begin{lemma}\label{lem:kt-is-dickson}
For $n\ge0$ and $k,t\in\mathbb Z$,
$$ F^{(k,t)}_{n+1}(X)=D^{(2)}_n(kX;-t),\qquad L^{(k,t)}_{n}(X)=D^{(1)}_n(kX;-t).$$
\end{lemma}


\begin{lemma}\label{lem:binet}
Let $R$ be a commutative ring and $x,\alpha,\beta,a $ belonging to $R$ satisfies $\alpha+\beta=x$ and
$\alpha\beta=a$.  Then for every $n\ge0$, 
\[
   D^{(1)}_n(x;a)=\alpha^n+\beta^n,\qquad
   (\alpha-\beta)\,D^{(2)}_n(x;a)=\alpha^{n+1}-\beta^{n+1}.
\]
If $\alpha=\beta$, so that $x=2\alpha$ and $a=\alpha^2$, then
\[
   D^{(1)}_n(x;a)=2\alpha^{n},\qquad
   \alpha\,D^{(2)}_n(x;a)=(n+1)\,\alpha^{n+1}.
\]
\end{lemma}

\begin{proof}
Keeping in mind the defining recurrence of the Dickson polynomials and the
hypothesis, the proof follows by induction on $n$.
\end{proof}

\begin{lemma}\label{lem:rescale}
Let $R$ be a commutative ring and let $a,c\in R$. Then, for every $n\ge0$ and
$j\in\{1,2\}$, $ D^{(j)}_n\bigl(cX;\;c^2a\bigr)=c^{\,n}\,D^{(j)}_n(X;a).$
In particular, if $c$ is a unit in $R$ and $x$ is a root of
$D_n^{(j)}(X;c^2a)$, then $x/c$ is a root of $D_n^{(j)}(X;a)$.
\end{lemma}

\begin{proof}
The result follows by induction on $n$ and the final assertion follows by evaluating the identity at $X=x/c$. 
\end{proof}



\noindent \textbf{Remark. }Fix $a\in\Z$ and $x\in\C$.  The {characteristic polynomial} of the recurrence
\eqref{eq:dickson} at the point $x$ is given by
\[
   \chi_{x,a}(Y):=Y^2-xY+a\ \in\C[Y].
\]
 Let $\alpha,\beta\in\C$ be its two roots. Then, the relations
 $\alpha+\beta=x,\qquad \alpha\beta=a,$
 tpgether with Lemma~\ref{lem:binet} gives the corresponding Dickson polynomials
 in terms of $\alpha$ and $\beta$.

\begin{lemma}\label{lem:nondeg} Let $a\ne0$ and let $x$ be a root of either
 $D_m^{(1)}(\,\cdot\,;a)$ or  $D_m^{(2)}(\,\cdot\,;a)$. Then the roots
$\alpha$ and $\beta$ of $\chi_{x,a}$ are distinct.
\end{lemma}

\begin{proof}
Since $\alpha\beta=a\ne0$, we have $\alpha\ne0$. If $\alpha=\beta$, Lemma~\ref{lem:binet}
gives $(m+1)\alpha^{m+1}=0$ or $2\alpha^m=0$, which is a contradiction.
\end{proof}

\begin{lemma}\label{lem:alpha}
Let $a\neq0$. Then the following hold:
\begin{enumerate}
\item\label{it:fst} If $D^{(1)}_m(x;a)=0$, then $\alpha^{2m}=-a^{m}$.
\item\label{it:sec} If $D^{(2)}_m(x;a)=0$, then $\alpha^{2(m+1)}=a^{m+1}$.
\end{enumerate}
\end{lemma}

\begin{proof}
Using Lemmas~\ref{lem:nondeg} and~\ref{lem:binet}, respectively, we have $\alpha\neq\beta$,~
$\alpha^m=-\beta^m$ and $\alpha^{m+1}=\beta^{m+1}$. Keeping in mind the relation, 
$\alpha\beta=a$, the proof is complete.

\end{proof}


\section{Proof of Theorem~\ref{thm:main} }\label{sec:main} 
\begin{proof}[Proof of Theorem~\ref{thm:main}]
By Lemma~\ref{lem:kt-is-dickson}, $x$ is a root of $F^{(k,t)}_{n}$ if and only if $kx$ is a
root of $D^{(2)}_{n-1}(\,\cdot\,;-t)$, and a root of $L^{(k,t)}_{n}$ if and only if $kx$ is a
root of $D^{(1)}_{n}(\,\cdot\,;-t)$. Keeping in mind Lemma~\ref{lem:rescale}, 
without loss of generality, we may assume that $k=s=1$, so  $a:=-t=\pm1$ and it
is enough to prove $\OO_{\Q(x)}=\Z[x]$.
 
 Let $\beta=x-\alpha$. By Lemma~\ref{lem:nondeg}, we have $\alpha\ne\beta$, and
 $\alpha\ne0$ since $\alpha\beta=a\ne0$. Now, we split the proof into two cases.

\noindent {Case 1: When $a=-1$, i.e., $t=1$.}  Here $\alpha\beta=-1$, so $\beta=-\alpha^{-1}$ and
$x=\alpha+\beta=\alpha-\alpha^{-1}$. Moreover,
$\alpha^2\neq-1$.  Suppose that $x$ is a root of $D_{n-1}^{(2)}$ with $n$ odd, or of $D_n^{(1)}$ with $n$ even. Then using Lemma~\ref{lem:alpha}, it is easy to see that $\alpha^{2n}=-1$. Applying Proposition~\ref{prop:critA} with $N=n$, we have $\OO_{\Q(x)}=\Z[x].$ Taking $\xi=\alpha$, we get $x=\xi-\xi^{-1}$.\\
\noindent {Case 2: When $a=1$ , i.e.,  $t=-1$.}   Here $\alpha\beta=1$, and hence
$x=\alpha+\alpha^{-1}$ and $\alpha^2\ne1$. Using Lemma~\ref{lem:alpha}, we see that
$\alpha^{2n}=1$ or $-1$, according as $x$ is a root of
$D_{n-1}^{(2)}$ or $D_n^{(1)}$. So, $\alpha$ is a root of unity, and $\OO_{\Q(x)}=\Z[x]$ by
Proposition~\ref{prop:critB}. Taking
$\xi=\alpha$ yields $x=\xi+\xi^{-1}$.

\end{proof}

\begin{proof} [Proof of Corollary~\ref{cor:CGH}]
Take $k=s=t=1$ in Theorem~\ref{thm:main}. Then $\theta=x$, and the
result follows immediately.
\end{proof}

\begin{definition}\label{def:cheb}
The {Chebyshev polynomials of the first and second kind} are defined by
$$ T_0=1,\quad T_1=X,\quad T_{n+2}=2XT_{n+1}-T_n$$ and
  $$U_0=1,\quad U_1=2X,\quad U_{n+2}=2XU_{n+1}-U_n  $$ respectively.
\end{definition}
Set $C_n:=D_n^{(1)}(X;1) \text{ and } S_n:=D_n^{(2)}(X;1).$ Then, it can be easily seen that $U_n(X)=S_n(2X),~ 2T_n(X)=C_n(2X),~n\ge0.$

\begin{corollary}\label{cor:cheb}
For any $n\ge1$, the following assertions hold:
\begin{enumerate}
\item Every irreducible factor of $S_n$ or $C_n$ is monogenic;
\item If $x$ is a root of $U_n$ or $T_n$, then
$\OO_{\Q(x)}=\Z[2x].$
\end{enumerate}
\end{corollary}

\begin{proof}
The first assertion follows from Theorem~\ref{thm:main} by taking $k=1$, $t=-1$,
and the second by taking $k=2,~t=-1$ together with the fact that
$U_n(X)=S_n(2X)$ and $2T_n(X)=C_n(2X)$.
\end{proof}


\end{document}